\documentclass[11pt, reqno]{amsart}

\usepackage {amsmath, amssymb, amscd, mathrsfs, url,verbatim,enumerate}
\usepackage[text={6.5in,9in},centering,letterpaper,dvips]{geometry}
\usepackage{multirow,latexsym,graphicx,epstopdf,comment}
\usepackage[all]{xy}

\usepackage[dvipsnames,usenames]{color}
\usepackage[colorlinks=true, urlcolor=NavyBlue, linkcolor=NavyBlue, citecolor=NavyBlue]{hyperref}
\usepackage{graphicx}
\usepackage{epsfig}
\usepackage[latin1]{inputenc}
\usepackage{amsmath}
\usepackage{amsfonts}
\usepackage{amssymb}
\usepackage{amsthm}
\usepackage{amscd}
\usepackage{verbatim}
\usepackage{subfigure}
\usepackage{pinlabel}
\usepackage{stmaryrd}
\usepackage{enumerate, enumitem}

\usepackage{tikz}
\usetikzlibrary{arrows}
\usetikzlibrary{decorations.pathreplacing}
\usepackage{verbatim}
	
\newtheorem{theorem}{Theorem}[section]

\newtheorem{lemma}[theorem]{Lemma}
\newtheorem{proposition}[theorem]{Proposition}

\newtheorem{corollary}[theorem]{Corollary}

\theoremstyle{definition}
\newtheorem{definition}[theorem]{Definition}

\theoremstyle{remark}
\newtheorem{remark}[theorem]{Remark}

\DeclareMathOperator{\barCP2}{\overline{\mathbb{CP}^2}}

\DeclareMathOperator{\xres}{X_{\mathrm{res}}}
\DeclareMathOperator{\xmf}{X_{\mathrm{MF}}}
\DeclareMathOperator{\mmin}{M_{\mathrm{min}}}

\subjclass[2013]{}

\author[Anar Akhmedov]{Anar Akhmedov}
\address {School of Mathematics, University of Minnesota, Minneapolis, MN, 55455, USA}
\email{akhmedov@umn.edu}

\author[\c{C}a\u{g}r{\i} Karakurt]{\c{C}a\u{g}r{\i} Karakurt}
\address{Department of Mathematics, Bo{\u{g}}azi{\c{c}}i University, Bebek, {\.{I}}stanbul, TR-34342, Turkey}
\email{cagri.karakurt@boun.edu.tr}

\author[S\"{u}meyra Sakall{\i}]{S\"{u}meyra Sakall{\i} }
\address{Department of Mathematical Sciences,
University of Arkansas,
Fayetteville, AR 72701, USA}
\email{ssakalli@uark.edu}

\numberwithin{equation}{section}

\title{Generalized Chain Surgeries and Applications}

\begin{document}
\maketitle

\begin{abstract}
We describe the Stein handlebody diagrams of Milnor fibers of Brieskorn singularities $x^p+y^q+z^r=0$. We also study the natural symplectic operation by exchanging two Stein fillings of the canonical contact structure on the links in the case $p=q=r$, where one of the fillings comes from the minimal resolution and the other is the Milnor fiber. We give two different interpretations of this operation, one as a symplectic sum and the other as a monodromy substitution in a Lefschetz fibration.
\end{abstract}

\section{Introduction}

In this paper, which is inspired by the recent work of Akhmedov and Katzarkov \cite{AkhLud}, 
we study a symplectic operation whose origins lie in the singularity theory. Let $f(x,y,z)$ be a polynomial in three complex variables which has an isolated singularity at the origin.  The  intersection of $\{f=0\}$ with the five-dimensional sphere of a small radius $\epsilon>0$ is a real three dimensional manifold called the link of the singularity. The set of complex tangencies forms a contact structure on the link called the canonical contact structure. 
The canonical contact structure admits two natural symplectic fillings, one coming from the minimal resolution and the other from the Milnor fiber $f(x,y,z)=\epsilon$. Except for simple singularites (of type ADE), these two symplectic fillings are nondiffeomorphic. Therefore if one of these fillings symplectically embeds into a closed symplectic 4-manifold $X$, one can define an interesting symplectic operation on $X$ by removing the embedded filling and replacing it by the other yielding a new symplectic 4-manifold. 

Here our focus is the symplectic surgery operation that comes from the Brieskorn singularity $x^p +y^q +z^r = 0$. Let us denote the corresponding Milnor fiber by $M(p,q,r)$. After giving preliminaries in Section \ref{Preliminaries}, in Section \ref{Milnor Fibers of Brieskorn Singularities} we describe a Lefschetz fibration on $M(p,q,r)$ compatible with its Stein structure. In Section \ref{Handlebody diagrams}, we give a detailed Stein handlebody description of $M(p,q,r)$. 

In Section \ref{OBs} we focus on the case $p = q = r$. We first elucidate the Lefschetz fibration on the minimal resolution $\mmin(p,p,p)$. Then we show that the open books on $\partial M(p,p,p)$ and $\partial \mmin(p,p,p)$ induced by the Lefschetz fibrations are the same. In Section \ref{Main Sec} we define our symplectic surgery operations which we call {\em Generalized Chain Surgeries}. The name comes from the mapping class group terminology. More specifically, we call the operation of cutting out the minimal resolution $\mmin(p,p,p)$ and gluing the Milnor fiber $M(p,p,p)$ in a symplectic 4-manifold {\em Generalized Chain Blow-up}, and the reverse operation {\em Generalized Chain Blow-down}. We also obtain a new relation in the mapping class group and via this relation we interpret the generalized chain surgeries as monodromy substitutions. This is in the same sprit as Endo, Mark, Van Horn-Morris' description  \cite{EMVHM} of symplectic rational blowdown as a monodromy substitution.
In Section \ref{Symp Sum}, we prove that the Generalized Chain Blow-up corresponds to the symplectic sum with a degree $p$ hypersurface in $\mathbb{CP}^3$. One can compare this with the following well-known fact: The symplectic rational blowdown of a sphere of self-intersection $-4$ corresponds to symplectic summing with $\mathbb{CP}^2$ along a quadratic curve. 

The last section is about the relation between our surgeries and the chain surgery. We show that the two chain substitutions correspond to exchanging the minimal resolutions with the Milnor fibers of Brieskorn singularities of types $(2,2g+1,4g+2)$ and $(2,2g+2,2g+2)$.

\subsection*{Acknowledgements}
This project has started on January 2019, while the first author was visiting Harvard University. AA would like to thank the Department of Mathematics at Harvard University for its hospitality. AA was partially supported by a Simons Research Fellowship and Collaboration Grant for Mathematicians from the Simons Foundation. Part of this work was done while \c{C}K visited SS at MPIM, and SS visited \c{C}K at IMBM. We would like to thank both MPIM and IMBM for their hospitality and support. We are all very grateful to L. Katzarkov for some helpful discussions. We would like to thank Marco Golla for his interest and comments on our work. We are also very grateful to the referee for constructive comments and suggestions that helped to improve the manuscript.

\section{Preliminaries}
\label{Preliminaries}
We assume the reader is familiar with the basics of contact topology such as Legendrian knots, Legendrian surgery, Stein and Strong fillings, open books, Giroux correspondence. Let us begin with singularities and their resolution graphs.

\subsection{Singularities and their resolution}

Let $g: \mathbb{C}^n \rightarrow \mathbb{C}$ be a function. We say $g$ is singular at $(y_1, \cdots, y_n)$ if $\partial g / \partial x_1 = 0, \cdots, \partial g / \partial x_n = 0$ at $(y_1, \cdots, y_n)$. Suppose now that $g:  (\mathbb{C}^n,0) \rightarrow (\mathbb{C},0) $ is a polynomial which is singular only at the origin. By restricting $g$ to a small ball around the origin, Milnor shows that $g$ defines a fibration outside of the zero locus. By intersecting the zero set $V_g := \{(x_1, \cdots x_n) : g(x_1, \cdots x_n)=0\}$ with the $(2n-1)$-sphere centered at the origin with small enough radius $\epsilon >0$, we get the link of the singularity. By perturbing the zero set with a small parameter $\epsilon >0$, we get a nonsingular set $g^{-1}(\epsilon)$ which is called the Milnor fiber of the singularity.

If there exists a smooth variety $\widetilde V$ and a map $\pi: \widetilde V \rightarrow V_g$ such that $\pi$ is a biholomorphic morphism from $\widetilde V \setminus  \pi^{-1}(0)$ to $V_g \setminus \{0\}$, we call the pair $(\widetilde V, \pi)$ a resolution of the singularity. By \cite{Hrn} every singularity admits a resolution, though it may not be unique.

In this paper we will be interested in the cases $n=2$ and $n=3$ only. When $n=2$, the link of a plane curve singularity is an honest link in $S^3$. For example the link of the singularity $x^p+y^q$ is the $(p,q)$-torus link. Milnor's fibration is a singular fibration over $B^4$ which has one singular fiber over the origin, the regular fibers $M(p,q)$ are genus minimizing Seifert surfaces of the $(p,q)$-torus link. One can resolve any plane curve singularity by successive blow-ups. The explicit method for obtaining resolution graphs of plane curve singularities is described in \cite{EisenbudNeumann}. 
Given $p,q$ let $N = \mathrm{gcd}(p,q)$, $p' = p/N, q' = q/N$ and consider continued fraction expansions 

\begin{equation}
\cfrac{p'}{q'} = a_0 - \cfrac{1}{a_1 - \cfrac{1}{a_2 -\cdots - \cfrac{1}{a_l}}}, \quad \cfrac{q'}{p'} = b_0 - \cfrac{1}{b_1 - \cfrac{1}{b_2 -\cdots - \cfrac{1}{b_k}}}
\end{equation}
\\
where $a_0,b_0 \geq 1$ and $a_i,b_j \geq 2$ for $i, j \neq 0$. Then the resolution graph of the link of the singularity $x^p+y^q$ is given by Figure \ref{ResGr}.
\begin{figure}[h] 
		\includegraphics[width=.6\textwidth]{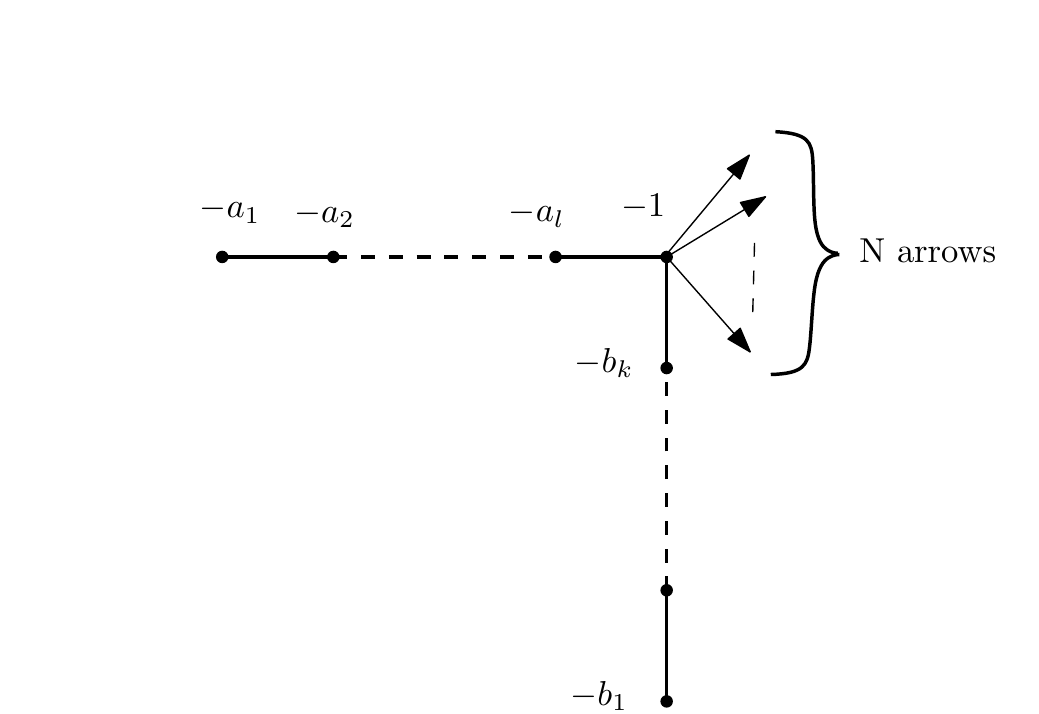}
		\caption{Resolution graph of the link of $x^p+y^q$}
		\label{ResGr}
	\end{figure}

When $n=3$, the link is a closed oriented 3-manifold and its Milnor fiber is a 4-manifold bounded by the link. If $\pi: \widetilde V \rightarrow V_f$ is any resolution, each irreducible component of exceptional divisor $\pi^{-1}(0)$ is a complex curve. We say that the resolution is \emph{minimal} if no irreducible component of its exceptional divisor is a rational curve of self intersection $-1$. We say that a resolution is \emph{good} if its exceptional divisor is a normal crossing divisor. The reader is advised to turn to \cite{Nem} for more information.




Suppose $f: \mathbb{C}^3 \rightarrow \mathbb{C}$ is singular at the origin. The set of complex tangencies of the link forms a contact structure on it, called the canonical contact structure. By \cite{CNP} the canonical contact structure is unique up to contactomorphism.

\subsection{Milnor open books and horizontal open books}

Now let us recall Milnor open books and horizontal open books.

\begin{definition}
Let $g : (\mathbb{C}^3, 0) \rightarrow (\mathbb{C}, 0)$ be a polynomial (with an isolated singularity at $0$) and $(V_g,0)$ is the singular space as above. The link of the singularity is $Y = V_g \cap S^5$. Now let us choose a non-constant function $f: (V_g,0) \rightarrow (\mathbb{C},0)$ holomorphic on $V_g \setminus {0}$. The open book decomposition of the link $Y$ with binding $L = Y \cap f^{-1}(0)$ and projection $\pi = f/{|f|} :Y \setminus L \rightarrow S^1 \in \mathbb{C}$ is called  \emph{the Milnor open book induced by $f$}.
\end{definition}

\begin{definition}
Let $Y$ be a plumbing of circle bundles over surfaces according to a good resolution graph with $s$ vertices. An open book decomposition of $Y$ is said to be \emph{horizontal} if all the pages are transverse to the fibers at each point in their interior, each component of the binding is a fiber and the orientation induced on the binding by the pages coincides with the orientation of the fibers induced by the fibration.
The vector $\mathbf{n} = (n_1, \cdots, n_s)$ where $n_j$ is the number of components of the binding lying as the fibers of the $j$-th vertex, is called the \emph{binding vector} of the horizontal open book. 
\label{horz}
\end{definition}

Milnor open books are all compatible with the natural contact structure on the boundary and they are horizontal when considered on the plumbing description of the Milnor fillable 3-manifold.

The resolution graph of a plane curve singularity naturally describes the link of the singularity as the binding of a horizontal open book of $S^3$, where the component of the binding vector at each vertex is the number of arrows connected to that vertex.

\begin{theorem} (\cite{CNP})
Let $Y$ be as in Definition \ref{horz}. Two horizontal open books on $Y$ are isomorphic if and only if they have the same binding vector.
\label{CNP thm}
\end{theorem}

\section{Milnor Fibers of Brieskorn Singularities}
\label{Milnor Fibers of Brieskorn Singularities}
Let $p$, $q$ and $r$ be integers $\geq 2$. Consider the complex hypersurface singularity 
$$f(x,y,z)=x^p+y^q+z^r=0.$$
For any $\epsilon \in \mathbb{C}$ which is small in norm,  we denote the Milnor fiber $f^{-1}(\epsilon)$ by $M(p,q,r)$. Our aim is to describe a compatible Lefschetz fibration on $M(p,q,r)$ compatible with its Stein structure. 

Consider the plane curve singularity 
$$g(x,y)=x^p+y^q.$$
The link of this singularity is the $(p,q)$ torus link in $S^3$.  By Milnor's theorem the map $(x,y)\to \frac{g(x,y)}{|g(x,y)|}$ is a fibration on $\{(x,y)\,:\,x^2+y^2\leq \epsilon \} \setminus\{g(x,y)=0 \}$ for small $\epsilon$. Denote the regular fibers by $M(p,q)$. In $S^3$ these regular fibers are isotopic to a genus minimizing Seifert surface of the $(p,q)$ torus link (see Figure~\ref{fig:PALF2}). To construct this surface we consider the braid description of the $(p,q)$ torus link which is a braid on $p$ strands corresponding to the positive word $(\sigma_1 \cdots \sigma_{p-1})^q$. Then we see the quasipositive Seifert surface by taking $p$ disks attaching a twisted band for each element appearing in the word, as shown in the figure for $p=q=3$ case.

By a small perturbation the unique singular fiber over the origin can be turned into $(p-1)(q-1)$ nodal singularities, and the above singular fibration can be turned into a positive allowable Lefschetz fibration on $B^4$ over the $2$-disk compatible with the standard Stein structure on $B^4$. The monodromy of this Lefschetz fibration is $\phi_{p,q}=\prod_{j=1}^{q-1} \prod_{i=1}^{p-1} t_{i,j}$ where $t_{i,j}$ is the right handed Dehn twist about the simple closed curve $\gamma_{i,j}$ as in Figure \ref{fig:PALF2}.  (\cite{AkOz}[Theorem 1])
\begin{figure}[h]
	\includegraphics[width=0.70\textwidth]{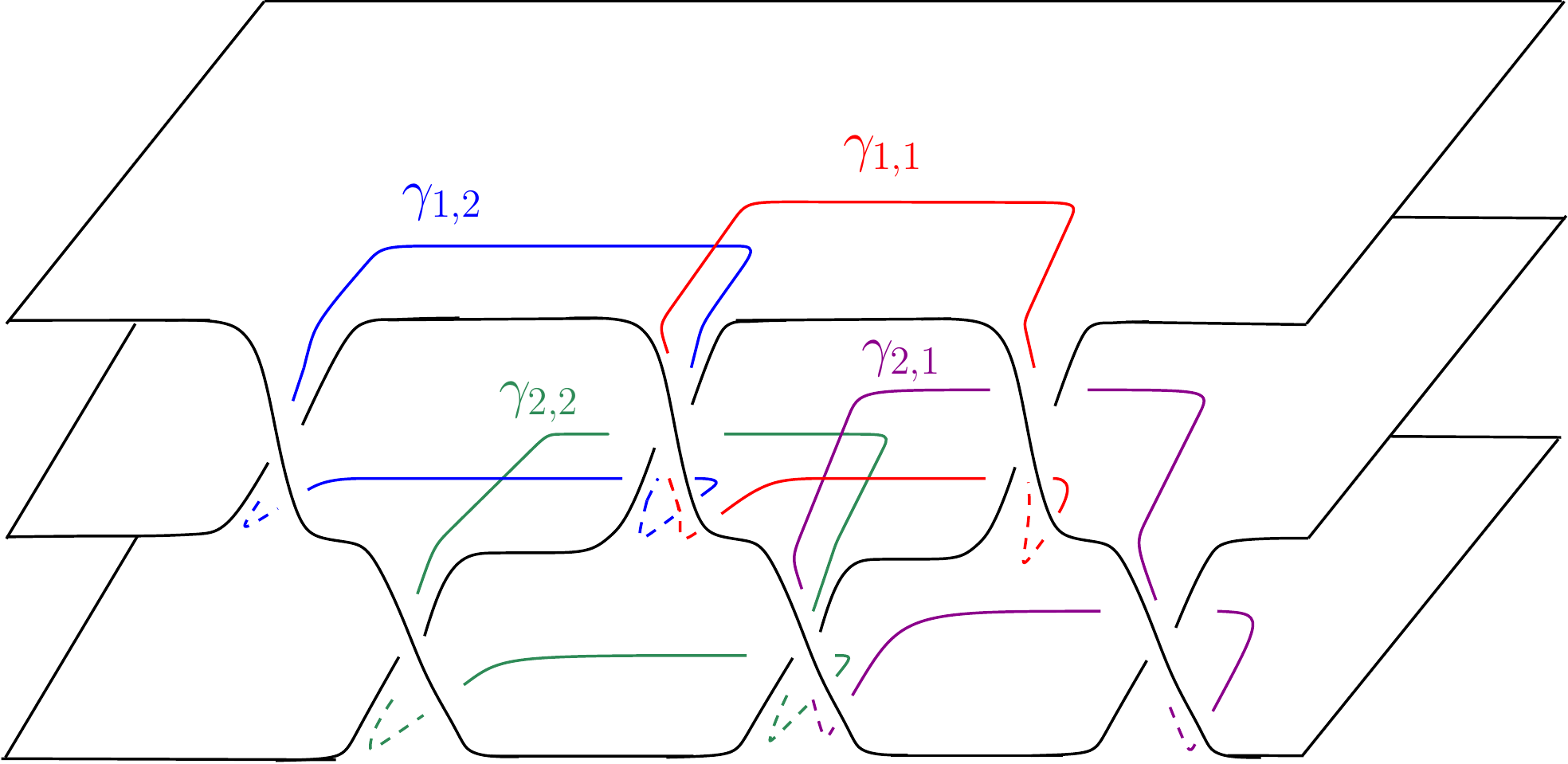}
	\caption{Genus minimizing Seifert surface of $(3,3)$ torus link with its vanishing cycles}
	\label{fig:PALF2}
\end{figure}

Now $M(p,q,r)$ is just an $r$-fold branched covering of $B^4$ branched along a regular fiber $M(p,q)$. The PALF on $B^4$ naturally gives rise to a PALF on $M(p,q,r)$ whose monodromy is $\phi_{p,q}^r$. With this information one can get quickly the smooth handlebody picture of $M(p,q,r)$  (compare \cite{Akb}, \cite{AkKir} and \cite{KirMel}). The Lefschetz fibration on $M(p,q,r)$ is obtained from the one on $B^4$ by attaching $(p-1)(q-1)(r-1)$ 2-handles along vanishing cycles according to the word $\phi_{p,q}^{(r-1)}$ with framing one less than the page framing. Notice that the exponent of $\phi_{p,q}$ is reduced by one to reconstruct $B^4$ with the PALF of $(p,q)$-torus knot. Clearly, each $\gamma_{i,j}$ bounds a disk in $S^3$, the page framing relative to this disk is $-1$ as $\gamma_{i,j}$ passes through two bands each making a left handed half twist. Hence the framing on each 2-handle is $-2$ relative to the Seifert framing. To sum up, we saw that $M(p,q,r)$ admits a handlebody diagram consisting of $(p-1)(q-1)(r-1)$ 2-handles each of which is attached on an unknot with framing $-2$. From this we recover the well-known fact that $M(p,q,r)$ has the homotopy type of wedge of ($\mu = (p-1)(q-1)(r-1)$) 2-spheres. Here $\mu$ is the Milnor number of the singularity and also the intersection form is even. 

Linking information of the 2-handles comes from the order of the Dehn twists. For example, after we attach a 2-handle on $\gamma_{1,1}$ we need to take a push-off $\gamma_{1,2}^+$ of $\gamma_{1,2}$ using the page framing. This will cause a linking between $\gamma_{1,1}$ and $\gamma_{1,2}^+$. Orienting each $\gamma_{i,j}$ counterclockwise we get the following linking information. 
\begin{align*} 
\mathrm{lk} (\gamma_{i,j}, \gamma_{i+1,j}^+) &=1, \; \text{for all } \; 1 \leq i \leq p-2, \; \text{for all } \; 1 \leq j \leq p-1,\\
\mathrm{lk} (\gamma_{i,j}, \gamma_{i,j+1}^+) &=1, \; \text{for all } \; 1 \leq i \leq p-1, \; \text{for all } \; 1 \leq j \leq p-2,\\
\mathrm{lk} (\gamma_{i,j}, \gamma_{i+1,j+1}^+) &=-1, \; \text{for all } \;  1 \leq i \leq p-2, \; \text{for all } \;  1 \leq j \leq p-2,\\
\mathrm{lk} (\gamma_{i,j}, \gamma_{i,j}^+) &=-1,  \; \text{for all } \; 1 \leq i \leq p-1, \; \text{for all } \; 1 \leq j \leq p-1.
\end{align*} 

The above information uniquely determines the linking between attaching circles of 2-handles. From this one can easily read off the intersection form. Let $h_{i,j}^k, \; 1\leq i \leq p-1, 1\leq j \leq q-1, 1\leq k \leq r-1$ denote the 2-handle corresponding to the vanishing cycle $\gamma_{i,j}$ in the level $k$. Each $h_{i,j}^k$ uniquely determines a homology class $[h_{i,j}^k]$ represented by an embedded sphere which is composed of the core disk of the 2-handle together with the trivial Seifert surface of $\gamma_{i,j}$ pushed into $B^4$. The homology classes $[h_{i,j}^k]$ form a basis for $H_2(M(p,q,r))$. The intersection data is as follows.
\begin{align*} 
[h_{i,j}^k] [h_{i,j}^k] &= -2,\; 1\leq i \leq p-1, 1\leq j \leq q-1, 1\leq k \leq r-1, \\
[h_{i,j}^k] [h_{i+1,j}^l] &= 1,\; 1\leq i \leq p-1, 1\leq j \leq q-1, 1\leq k \leq r-1, 1\leq l \leq r-1,\\
[h_{i,j}^k] [h_{i,j+1}^l] &= 1,\; 1\leq i \leq p-1, 1\leq j \leq q-1, 1\leq k \leq r-1,1\leq l \leq r-1,\\
[h_{i,j}^k] [h_{i+1,j+1}^l] &= -1,\; 1\leq i \leq p-1, 1\leq j \leq q-1, 1\leq k \leq r-1,1\leq l \leq r-1,\\
[h_{i,j}^k] [h_{i,j}^l] &=-1,\; 1\leq i \leq p-1, 1\leq j \leq q-1, 1\leq k \leq r-1,1\leq l \leq r-1.\\
\end{align*} 

For example when $r=2$, the intersection form is given by the incidence matrix of the $(p-1)\times (q-1)$ hexagonal sphere packing as in Figure \ref{fig:pack2}. The vertices of the packing represent $[h_{i,j}]$. Here we drop the super scripts $k$, as $r=2$ the only possible value for $k$ is 1. The edges labeled with $\pm$ signs indicate intersection numbers, the index $i$ increases from right to left and the index $j$ increases from top to bottom. Note that only the edges with positive slopes have negative labels. We would like to point out that our description of the intersection form agrees with that of \cite{Gabr}.
	\begin{figure}[h] 
		\includegraphics[width=.50\textwidth]{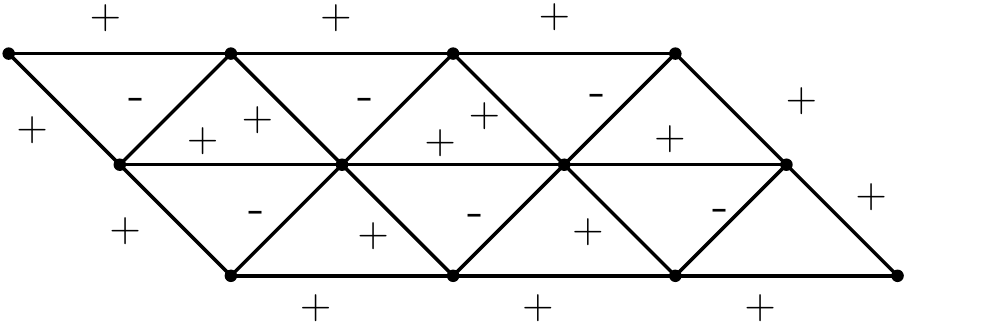}
		\caption{The $(3\times 4)$ sphere packing representing the intersection form of $M(4,5,2)$.}
		\label{fig:pack2}
	\end{figure}

\section{Handlebody diagrams}
\label{Handlebody diagrams}
Next we show how to convert the PALF described in the previous section, into a Stein handlebody diagram.  We shall describe an explicit algorithm giving Stein handlebody picture of $M(p,q,r)$ and as a byproduct we get the contact surgery diagram of the canonical contact structure on the  Brieskorn sphere $\Sigma (p,q,r)$. The latter was also given by the work of Harvey-Kawamuro and Plamenevskaya in \cite{HarKawPla}. In fact our argument is similar to theirs but we need to  adapt it to our $4$-dimensional setting. 
\begin{theorem}
$M(p,q,r)$ with the canonical Stein structure is given by a handlebody diagram consisting of $(p-1)(q-1)(r-1)$ Legendrian 2-handles attached along standard Legendrian unknots. 
\label{thm4.1}
\end{theorem}

\begin{proof}
	This can be seen as the  reverse Akbulut-Ozbagci algorithm (\cite{AkOz}). Since the monodromy $\phi_{p,q}$ describes a PALF of $B^4$, we can see the rest of the $(p-1)(q-1)(r-1)$ vanishing cycles (belonging to the factorization of $\phi_{p,q}^{r-1})$ as non-separating curves lying on different pages of an open book decomposition of $S^3$ supporting to the standard contact structure of $S^3$. According to the classification of Legendrian unknots given by Eliashberg and Fraser \cite{EliFra} each vanishing cycle is the standard Legendrian unknot. To get the complete handlebody description of $M(p,q,r)$, we need to see these Legendrian unknots explicitly. For this,  we appeal to Plamenevskaya's refinement \cite[Proposition 4]{Pl}  (see also \cite{HarKawPla}) of Akbulut Ozbagci algorithm where it was shown that the curves $\gamma_{i,j}$'s are in fact Legendrian unknots and we notice that the alternative picture of the Seifert surface of the $(p,q)$ torus knot has the property that the page framing agrees with the contact framing (see Figure \ref{SeifertS}). We see that each vanishing cycle $\gamma_{i,j}$ is an unknotted Legendrian knot in $(S^3,\xi_{\mathrm{std}})$ with Thurston-Bennequin number $-1$. We push the vanishing cycles appearing in the factorization of $\phi_{p,q}^{r-1}$ in different pages using the page framing which agrees with the contact framing. The linking information also agrees with the one in the hexagonal sphere packing diagram as in Figure \ref{fig:pack2} with $(p-1)$ rows and $(q-1)$ columns.

\begin{figure}[h] 
		\includegraphics[width=.7\textwidth]{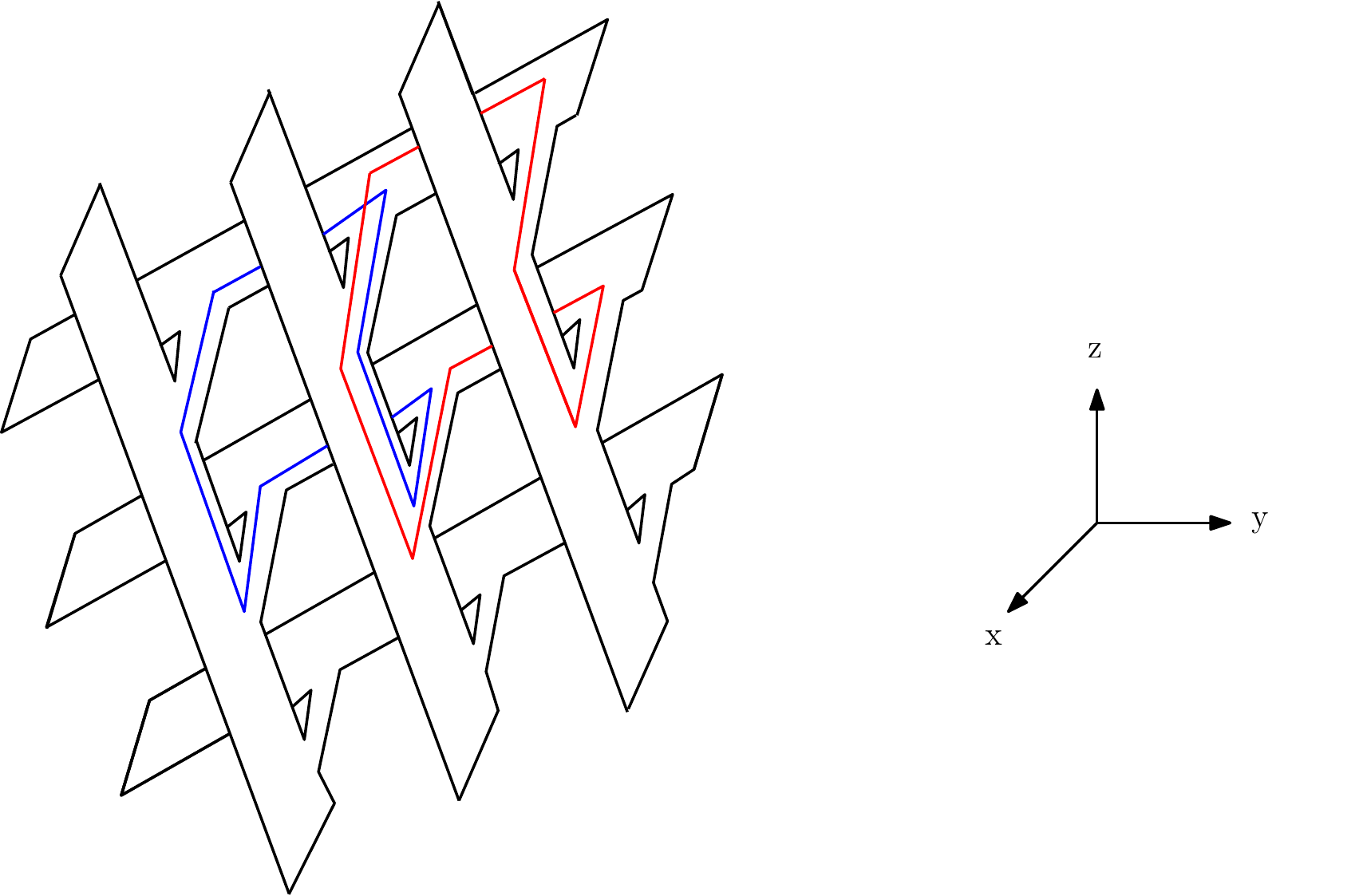}
		\caption{Alternative picture of the Seifert surface of $(3,3)$ torus link}
		\label{SeifertS}

	\end{figure}

The model for $+$ and $-$ linking are shown in Figure \ref{fig:Pm}.  Here the overcrossing and undercrossing information should be inferred from the front projection conventions. We apply this recipe and turn the sphere packing diagram to a Legendrian link diagram in $S^3$ and attach Legendrian 2-handles to attach each component. This gives the Stein handlebody picture for the double branched cover.  When $r\geq 3$, to get $r$-fold branched covers take the entire link diagram and create $r-2$ Legendrian push-offs and attach Stein handles.

		\begin{figure}[h] 
		\includegraphics[width=.40\textwidth]{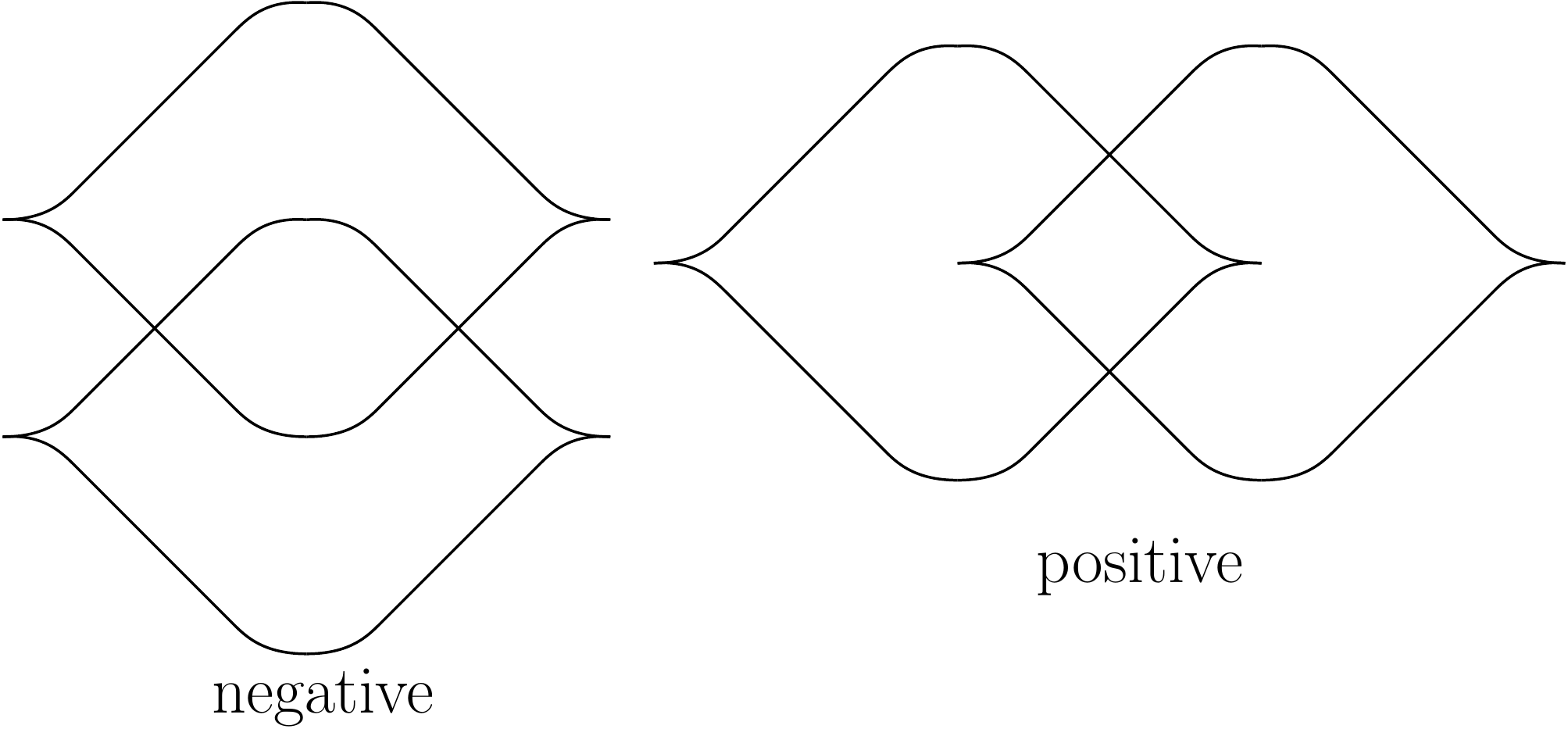}
		\caption{Linking conventions: Negative linking on the left hand side, positive on the right hand side. }
		\label{fig:Pm}
	\end{figure}

		\begin{figure}[h] 
	\includegraphics[width=.70\textwidth]{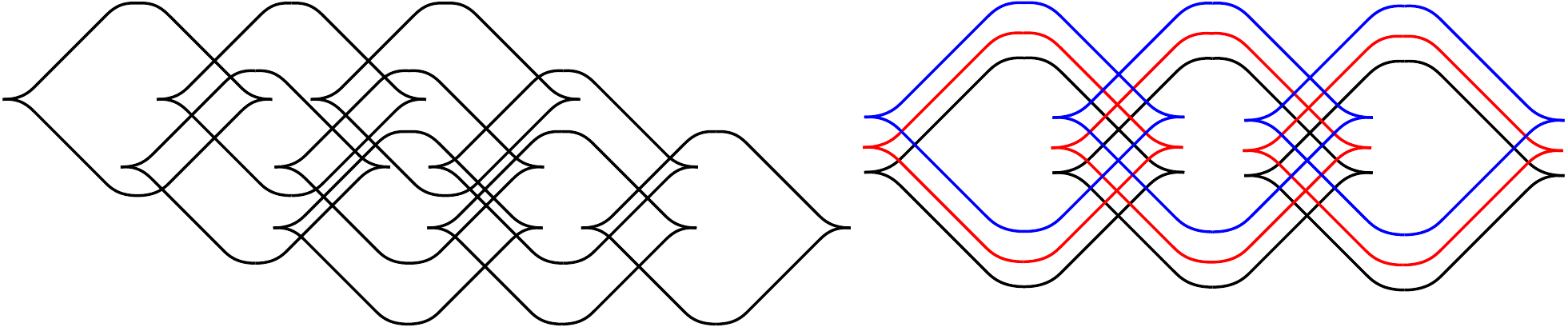}
	\caption{Stein handlebody pictures of $M(4,4,2)$}
	\label{fig:St1}
\end{figure}
\end{proof}

One of the inputs in our algorithm was expressing $M(p,q,r)$ as the $r$-fold branched cover over $B^4$. Similarly we can regard $M(p,q,r)$ as $p$-fold and respectively $q$-fold branched covers over $B^4$. These three descriptions in general yield different Stein handlebody diagrams (see for example figures \ref{fig:St1}, \ref{fig:St2}). It was shown by Kirby and Melvin (\cite{KirMel}) that the diagrams describe diffeomorphic $4$-manifolds in the case $(p,q,r)=(2,3,5)$ by a sequence of explicit handlebody moves. It would be interesting to know whether the same moves can be performed in Stein category.  

		\begin{figure}[h] 
	\includegraphics[width=1.0\textwidth]{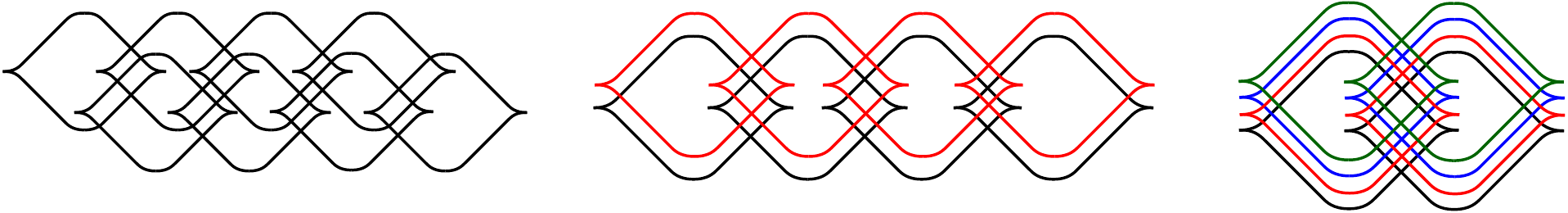}
	\caption{Stein handlebody pictures of $M(2,3,5)$}
	\label{fig:St2}
\end{figure}

\begin{corollary}
$c_1(M(p,q,r))=0$.   
\end{corollary} 
\begin{proof}
From the description above we see that $M(p,q,r)$ has the homotopy type of $\displaystyle \vee_{(p-1)(q-1)(r-1)}S^2$. Each Legendrian $2$-handle corresponds to a generator of $H_2(M(p,q,r);\mathbb{Z})$. According to Gompf (\cite{G2}), the evaluation of the first Chern class at each of these generators is equal to the rotation number of the attaching circle. Since our attaching circles are standardard Legendrian unknots their rotation numbers are all zero.

Alternatively, this corollary follows from the fact that the holomorphic tangent bundles of the smoothings of hypersurface singularities are trivial. Indeed, we know that all hypersurface singularities are Gorenstein (\cite{Durfee}), and Lemma 1.1 in \cite{Durfee} and Lemma 2 in \cite{Seade} prove Conjecture 1.6 in \cite{Durfee}. 
\end{proof}

\begin{theorem}
Suppose $p \leq p', q \leq q'$ and $r \leq r'$ then there exist a Stein embedding of $M(p,q,r)$ inside $M(p',q',r')$.
\end{theorem}

\begin{proof}
The smooth version of this theorem is well-known (e.g. see \cite{GS}). For the Stein version, our description implies that by adding Stein 2-handles to $M(p,q,r)$ one can get $M(p',q',r')$.
\end{proof}

\begin{remark}
Marco Golla communicated to us an alternative method of proving the above result using the deformation $x^p+y^p+z^p+t(x^{p'}+y^{p'}+z^{p'})$ which gives rise to a Stein cobordism from the canonical contact structure of $x^p+y^p+z^p$ to the canonical contact structure of $x^{p'}+y^{p'}+z^{p'}$ [\cite{EG}, Section 5].
\end{remark}

\section{Minimal Resolutions and Matching of Open Books}
\label{OBs}
Now we focus on the case where $p = q = r$. Let $\mmin(p,p,p)$ denote the space of the minimal resolution of the singularity $x^p+y^p+z^p$. We first discuss the Lefschetz fibration on $\mmin(p,p,p)$. Then we prove that the open books on $\partial M(p,p,p)$ and $\partial \mmin(p,p,p)$ induced by the Lefschetz fibrations are the same (Theorem \ref{thm:match}). First let us prove the following lemma which is needed in the sequel.

\begin{lemma}
	The minimal resolution graph of $x^p+y^p+z^p=0$ consists of a single vertex corresponding to a surface of genus $(p-1)(p-2)/2$ and self intersection $-p$.
\end{lemma}

\begin{proof}
We apply the algorithm described in Chapter III, Appendix 1 of \cite{Nem} to get the resolution graph of $x^p+y^p+z^p$. First notice that the resolution graph of the plane curve singularity $x^p+y^p=0$ is given by Figure \ref{ResG}.
	\begin{figure}[ht]
		\includegraphics[width=0.30\textwidth]{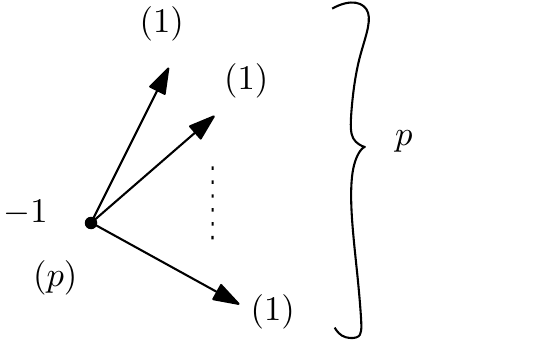}
		\caption{Resolution graph of $x^p+y^p=0$}
		\label{ResG}
	\end{figure}
	This is because $x^p+y^p$ can be factorized into $p$ linear factors each of which describe a line intersecting at the origin. A single blow-up at the origin separates these lines. The multiplicity of the exceptional divisor is $p$. Each arrowhead corresponds to the proper transform of a line (having multiplicity 1). Now take the $p$-fold brached cover branched along the lines, the exceptional sphere lifts to a closed surface self intersection $-p$. A simple Euler characteristic computation shows that the genus of this surface is as we claimed. 
\end{proof}

By a theorem of Bogomolov and  de Oliveira \cite{BO}, there is a deformation of the holomorphic structure on $\mmin(p,p,p)$ which gives a Stein filling of the canonical contact structure on the link. Moreover $\mmin(p,p,p)$ contains a symplectic surface of genus $g=(p-1)(p-2)/2$ with self intersection $-p$. The following result says that if another symplectic 4-manifold contains such a symplectic surface then a neighborhood of this surface is symplectomorphic to $\mmin(p,p,p)$ and it is possible to determine the vanishing cycles of a compatible Lefschetz fibration. 

\begin{proposition}(\cite{GayMark, ParkStip})
Suppose a symplectic 4-manifold $X$ contains a symplectic surface $\Sigma$ of genus $g=(p-1)(p-2)/2$ with self intersection $-p$. Then every neighborhood of $\Sigma$ contains a sub-neighborhood $N(\Sigma)$ with strongly convex boundary, that admits a symplectic Lefschetz fibration $\pi: N(\Sigma) \rightarrow D^2$ with $p$ vanishing cycles where the fibers are genus $g$ with $p$ boundary components. The monodromy is the product of right handed Dehn twists at each boundary component. The contact structures induced from the symplectic filling and the open book are both the same as the canonical contact structure. 
\label{GM-PS}
\end{proposition}

\begin{proof}
	This proposition is a special case of Theorem 1.1 in \cite{GayMark}. The last statement follows from \cite{ParkStip}.
		\begin{figure}[h] 
	\includegraphics[width=0.5\textwidth]{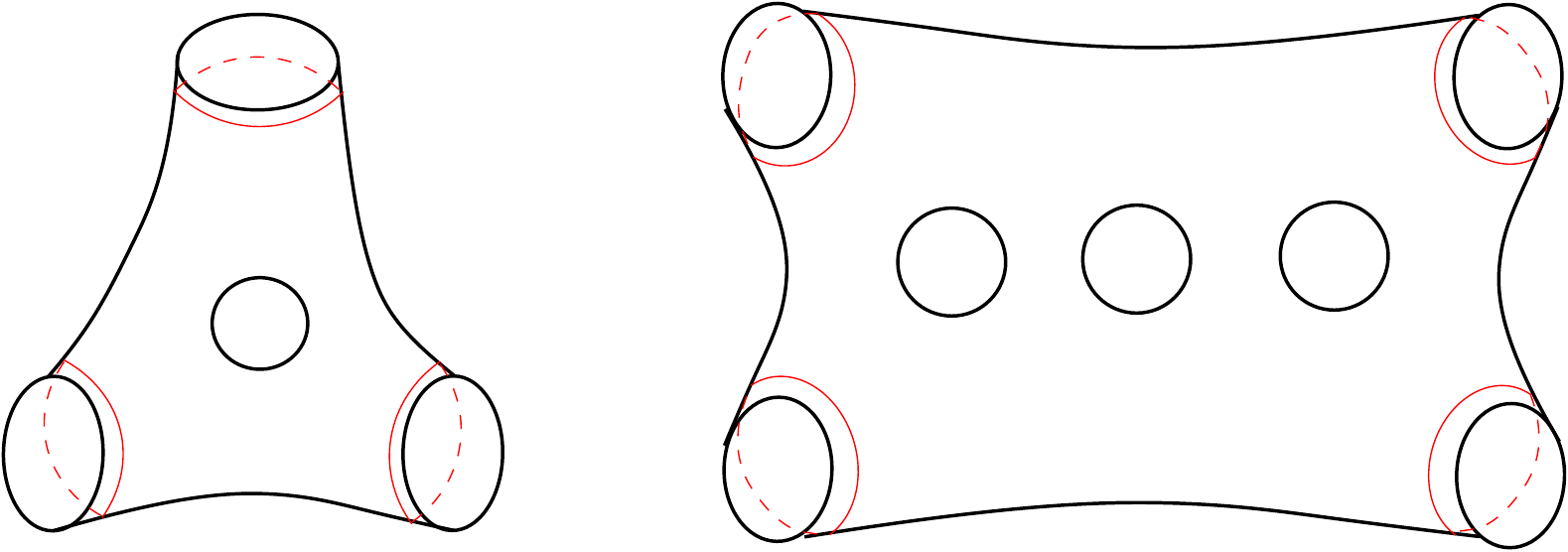}
	\caption{$p=3$ and $p=4$ cases}
	\label{}
\end{figure}

\end{proof}

\begin{theorem}
The open books on $\partial \mmin(p,p,p)$ and $\partial M(p,p,p)$ induced by the Lefschetz fibrations constructed in Proposition \ref{GM-PS} and respectively Section 3 are the same.
\label{thm:match}
\end{theorem}

\begin{proof}
By Theorem \ref{CNP thm} it suffices to show that both open books are horizontal with the same binding vector $\mathbf{n} = (p)$. For $\partial \mmin(p,p,p)$, this was explicitly shown in \cite{EtgOzb}. The result for $\partial M(p,p,p)$ is a consequence of branched cover description as follows. When we resolve the plane curve singularity $x^p+y^p =0$ we noticed that the link is a union of $p$ copies of the fibers of the Hopf fibration on $S^3$ (see Figure \ref{ResG} where each arrow represents a component of the link). When we take $p$-fold branched cover along the link, the Hopf fibration on $S^3$ gives rise to $S^1$ bundle over $\Sigma_g$ with Euler number $-p$ (where $g = (p-1)(p-2)/2$). To see the $S^1$ action upstairs, notice that $p$ distinct fibers in the branch locus descend to $p$ distinct points in the quotient $S^3 / S^1$ which is diffeomorphic to $S^2$. The branch cover of $S^2$ branched along $p$ distinct points is $\Sigma_g$. Clearly, we can take the branched cover in an $S^1$ equivariant way so that each $S^1$ fiber except those in the branch locus, lifts to $p$ distinct $S^1$ fibers. Moreover, the binding remains to be $p$ copies of the fiber. Hence the open book is horizontal with binding vector $\mathbf{n} = (p)$. In the level of open books, the branched cover description above corresponds to taking the $p$-th power of the monodromy of the $(p,p)$ torus link in $S^3$ whose page is a surface of genus $g$ with $p$ boundary components.

\end{proof}

\section{Generalized Chain Surgeries}
\label{Main Sec}
Using the ideas in the previous sections we define our symplectic surgery operation which we call Generalized Chain Surgeries (GC Surgeries in short) as follows.
\begin{definition}
Fix an integer $p \geq 3$. Suppose $X$ is a closed symplectic 4-manifold which contains a symplectic surface $\Sigma$ of genus $g = (p-1)(p-2)/2$ with self intersection $-p$. By Proposition \ref{GM-PS}, a neighborhood $N(\Sigma)$ of $\Sigma$ in $X$ is symplectomorphic to the minimal resolution of the singularity $x^p+y^p+z^p=0$ so that $N(\Sigma)$ forms a strong symplectic filling for the canonical contact structure on the link of the singularity. We remove $N(\Sigma)$ and glue in the Milnor fiber $M(p,p,p)$. The resulting manifold 
\begin{equation*}
X' = (X \setminus N(\Sigma)) \cup M(p,p,p)
\end{equation*}

\noindent admits a symplectic structure extending the one on the complement of $\Sigma$ in $X$. We call the operation $X  \dashrightarrow X'$ {\em Generalized Chain Blow-up (GC Blow-up)} and the reverse operation $X'  \dashrightarrow X$ {\em Generalized Chain Blow-down (GC Blow-down)}. 
\end{definition}

\begin{remark}
Note that the calculation of $b_2$ of the Milnor fiber $M(p,p,p)$ directly follows from Theorem \ref{thm4.1}, and that $b_2$ is larger for $M(p,p,p)$; hence the name GC Blow-up.
\end{remark}

Let us state a result which is an immediate consequence of our technique.
\begin{corollary}
For every $p \geq 3$, the canonical contact structure of the singularity $x^p+y^p+z^p$ has at least $2(p-2)$ distinct Stein fillings.
\end{corollary}
\begin{proof}
For every $3 \leq p' \leq p$, we can embed $M(p',p',p')$ into $M(p,p,p)$ and replace it by $\mmin (p',p',p')$. We can distinguish these fillings by their $b_2$.
\end{proof}
The above number of fillings is not optimal. When $p$ is large, one can embed different Milnor fibers $M(p',p',p')$ into $M(p,p,p)$ and get different fillings. In our case we only see finite number of Stein filings. On the other hand, in \cite{OO}, Ohta Ono prove that there are infinitely many Stein fillings of these canonical contact structures.

Now, let us discuss an equivalent formulation of GC blow-ups and blow-downs using monodromy substitutions of Lefschetz fibrations. Let $\Sigma_g^p$ be a surface with genus $g= (p-1)(p-2)/2$ with $p$ boundary components. Inside $\Sigma_g^p$ one can find a system $\Gamma$ of simple closed curves

\begin{align*} 
\Gamma =  \{&\gamma_{1,1}, \cdots , \gamma_{1,p-1}, \\
 & \gamma_{2,1}, \cdots , \gamma_{2,p-1}, \\
&\vdots\\
&\gamma_{p-1,1}, \cdots , \gamma_{p-1,p-1}\},
\end{align*} 

which intersect according to the $(p-1)\times (p-1)$ hexagonal sphere packing grid as in Figure \ref{fig:pack2}. From the grid we see the following intersection pattern.
\begin{align*} 
\gamma_{i,j} \cap \gamma_{i+1, j} =1, \;\; \text{for all } i= 1, \cdots, p-2, \text{ for all } j= 1, \cdots, p-1, \\
\gamma_{i,j} \cap \gamma_{i, j+1} =1, \;\; \text{for all } i= 1, \cdots, p-1, \text{ for all } j= 1, \cdots, p-2, \\
\gamma_{i,j} \cap \gamma_{i+1, j+1} =-1,\;\;  \text{for all } i= 1, \cdots, p-2, \text{ for all } j= 1, \cdots, p-2. \\
\end{align*} 
All the other intersections are zero. We call this system $\Gamma$ a $(p-1) \times (p-1)$ packing. An explicit realization of a $(p-1) \times (p-1)$ packing can easily be seen in the quasipositive Seifert surface of the $(p,p)$ torus link as in Figure \ref{fig:PALF2} where $p=3$ case is depicted. The surface $\Sigma_g^p$ deformation retracts onto $\bigcup \Gamma$. 

Conversely if one takes an arbitrary surface $F$ containing a system $\widetilde \Gamma$ of simple closed curves which forms a $(p-1) \times (p-1)$ packing then a regular neighborhood of $\widetilde \Gamma$ inside $F$ is homeomorphic to the surface $\Sigma_g^p$.

Let $t_{i,j}$ denote the right handed Dehn twist around $\gamma_{i,j}$ in $\Sigma_g^p$. Consider the word 
\begin{equation}
w= (t_{1,1}  \cdots  t_{1,p-1})(t_{2,1} \cdots t_{2,p-1}) \cdots (t_{p-1,1} \cdots t_{p-1, p-1}).
\end{equation}
Note that the Lefschetz fibration over $B^2$ with fibers $\Sigma_g^p$ corresponding to the word $w$ has total space $B^4$ and the induced open book on $S^3$ has binding the $(p,p)$ torus link. The word $w^p$ gives a Lefschetz fibration on $M(p,p,p)$ by Section 3. 
 
Let $t_{\delta_1}, \cdots, t_{\delta_p}$ denote the Dehn twists about the boundary parallel curves $\delta_1, \cdots \delta_p$ in each boundary component. The Lefschetz fibration corresponding to the word 
\begin{equation}
w_0 = t_{\delta_1} \cdots t_{\delta_p}
\end{equation}
is the one on $\mmin(p,p,p)$. Hence by  Theorem \ref{thm:match}, the open books match, in particular the monodromies agree and we have the following. 

\begin{theorem}
Suppose $\Gamma = \gamma_{i,j}$ is a system of simple closed curves in $\Sigma_g^p$ forming a $(p-1) \times (p-1)$ packing. Then the following relation holds in the mapping class group
\begin{align}\label{eq:rel}
\left ( \prod_{i=1}^{p-1} \prod_{j=1}^{p-1} t_{i,j}\right )^p = \prod_{k=1}^p \delta_k
\end{align}
\noindent where $t_{i,j}$ is the right handed Dehn twist about the simple closed curve $\gamma_{i,j}$ and $\delta_k$ is the curve parallel to $k$th boundary component.
\end{theorem}

\begin{remark}
We verified this identity using contact geometry and singularity theory. It would be interesting to see a proof using only mapping class group techniques.
\end{remark}

Let $X$ be a symplectic 4-manifold which admits a compatible Lefschetz fibration with fiber $F$. Suppose in $X$ there is a system $\widetilde \Gamma$ of circles which is a $(p-1) \times (p-1)$ packing. A neighborhood of $\widetilde \Gamma$ in $F$ is homeomorphic to $\Sigma_g^p$. Assume moreover that the monodromy of $X$ contains a word $w^p$. Then GC blow-down of $X$ corresponds to substituting $w_0$ in place of $w^p$. Conversely, if one finds a subsurface $\Sigma_g^p$ of the fiber $F$ and product of Dehn twists along the boundary of $\Sigma_g^p$ appears in the monodromy of $X$, then one can GC blow-up $X$ by substituting $w^p$ in place of $w_0$.

\section{Description as a Symplectic Sum}
\label{Symp Sum}
Using the ideas in the previous section, we will apply our symplectic surgery operation in some complex surfaces. Throughout let $g(x,y) =x^p+y^p$. The projective curve  $\Phi$ defined by the equation $g(x,y)=0$ is singular at one point $p_0$ in $\mathbb{CP}^2$. Let $\pi: X \to \mathbb{CP}^2$ denote the $p$-fold branched cover of $\mathbb{CP}^2$ with branch locus $\Phi$. The surface $X$ has an isolated singularity at $\pi^{-1}(p_0)$. Let $B$ be a small ball around $p_0$ and remove $\pi^{-1}(B)$ from $X$. Denote the resulting manifold $X'$. The boundary of $X'$ is $p$-fold branched cover of $S^3$ where the branch locus is the $(p,p)$ torus link which is the link of the singularity $g(x,y)=0$. 

The manifold $X'$ has a symplectic structure with concave boundary $\partial X'$ which is the link manifold of the surface singularity $z^p = g(x,y)$  equipped with the canonical contact structure. Then we can get two smooth symplectic manifolds $\xres$ and $\xmf$ by filling in the minimal resolution $\mmin$ and respectively the Milnor fiber $M(p,p,p)$ of the singularity $z^p = g(x,y)$. 

The passage $\xmf \dashrightarrow \xres$ can be seen as the removal of a symplectically embedded Milnor fiber and replacing it by the neighborhood of a symplectic surface of genus $(p-1)(p-2)/2$ with self intersection $-p$. In this sense, this operation is GC blowdown. Symbolically we have
\begin{equation*}
\xres = (\xmf \setminus M(p,p,p)) \cup \mmin.
\end{equation*}
Another interesting observation is that $\xmf$ can be seen as $p$-fold branch cover of $\mathbb{CP}^2$ branched along (the projectivization of) nonsingular curve $x^p+y^p = \epsilon$. This could be useful in detecting the diffeomorphism type of $\xmf$. As explained in \cite{GS} p.235, such a branched cover is diffeomorphic to the degree $p$ hypersurface $S_p$ in $\mathbb{CP}^3$. Hence we see $M(p,p,p)$ is embedded inside $S_p$ as a complex submanifold.

As an example, let us discuss the family of hypersurfaces $S_p$ of degree $p$ in $\mathbb{CP}^3$, which will be used as a building block in our construction. Let us recall that the hypersurface $S_p$ is a smooth, simply connected, complex surface and can be obtained as branched cover over $\mathbb{CP}^2$ branched along the divisor $pH$ \cite{Mandelbaum}.  For a degree $d$ cyclic branched cover $X \rightarrow Y$ with branch locus $B$, one has the following formulas:
\begin{equation*}
e(X) = de(Y) - (d-1)e(B), \quad \sigma(X) = d \sigma (Y) - ((d^2-1)/3d) B^2. 
\end{equation*}

By applying these formulas for $S_p$, we have
\begin{equation*}
e(S_p) = 3p - (p-1)(2 - 2(p-1)(p-2)/2) = 3p - (p-1)(-p^2 +3p) = p^3 -4p^2 + 6p,
\end{equation*}
\begin{equation*}
\sigma(S_p) = p - ((p^2-1)/3p) p^2 = p - (p+1)(p-1)p/3. 
\end{equation*}

Moreover, $S_p \# p (\barCP2)$ admits a Lefschetz fibration over $S^2$. For the sake of completeness, let us review the construction of this Lefschetz fibration and provide the details, which will be helpful to the reader. We thank the referee for some suggestions along these lines. The Lefschetz pencil structure on $S_{p}$ comes from pulling back the standard pencil of lines on $\mathbb{CP}^2$ via the branched covering. Since the generic fiber $F$ of this pencil is the branched cover of the line in $\mathbb{CP}^2$ branched over $p$ points, the fiber genus can be computed using Riemann-Hurwitz formula:
\begin{equation*}
2 - 2g(F) = p (e(S^2)) - p(p-1) = -p^2 + 3p, \;\mathrm{thus} \; g(F) = (p-1)(p-2)/2.
\end{equation*}

Notice that one base point of the pencil of lines in $\mathbb{CP}^2$ pulls back to $p$ base points in $S_p$, which one can blow-up to get the fibration $S_p \# p (\barCP2)$. Using this description of the pencil, we can compute the
number of vanishing cycles of the resulting genus $(p-1)(p-2)/2$ Lefschetz fibration on $S_p \# p (\barCP2)$, either by determining the number of tangencies of the degree $p$ curve with the linear fibers of the pencil of lines $\mathbb{CP}^2$, or by comparing the Euler characteristic of  $S_p \# p (\barCP2)$ with the Euler characteristic of the $F \times S^2$:

\begin{equation*}
e(S_p \# p (\barCP2)) - e(F) \cdot e(S^2)= p^3 - 4p^2 +7p - 2(-p^2+3p) = p(p-1)^2. 
\end{equation*}

Thus, the total number $n$ of vanishing cycles of the given fibration $\Sigma_{(p-1)(p-2)/2} \rightarrow S_p \# p (\barCP2) \rightarrow S^2$ is $n= p (p-1)^2$. The following table illustrates various cases 
\\

\begin{center}
 \begin{tabular}{|p{1cm}|p{2cm}|p{4.5cm}|p{5cm}|} 
 \hline
 \multicolumn{4}{|c|}{Lefschetz fibration on $S_p \# p (\barCP2)$} \\
[0.5ex] 
 \hline
 $p$ & Fiber genus & Number of singular fibers & Total space of the fibration \\ [0.5ex] 
 \hline
 2 & 0 & 2 & $\mathbb{CP}^2 \# 3\barCP2$ \\ [0.5ex] 
 \hline
 3 & 1 & 12 & $\mathbb{CP}^2 \# 9\barCP2$\\[0.5ex] 
 \hline
 4 & 3 & 36 & $K3 \# 4\barCP2$ \\[0.5ex] 
 \hline
 5 & 6 & 80 & Complex surface of general type \\
[0.5ex] 
 \hline
\end{tabular}
\end{center}
\vspace{0.5cm}

The diffeomorphism types of the total spaces of these fibrations for low values of $p$, which we included in the table above, can be found in \cite{GS} (see discussion on pages 23-24).

Now we will illustrate what our surgery operation $\xmf= S_p \dashrightarrow \xres$ does for various choices of $p$. 
When $p=2$, the singularity $x^2+y^2+z^2=0$ is of type $A_1$ which is a rational double point. Hence the Milnor fiber $M(2,2,2)$ is diffeomorphic to the minimal resolution which is a $D^2$ bundle over $S^2$ with Euler number -2. We have $S_2 = S^2 \times S^2$ but the surgery operation  $S_2 \dashrightarrow \xres$ does not change the diffeomorphism type.

When $p=3$, $S_3$ is diffeomorphic to $\mathbb{CP}^2 \# 6\barCP2$. We need to describe the complement of $M(3,3,3)$ in $S_3$. As we described in Section 5, the PALF of $M(3,3,3)$ corresponds to the left hand side of the star relation in Equation \ref{eq:rel}. When we cap off the three boundary components of the fibers of this PALF and glue in $T^2 \times D^2$, we get a Lefschetz fibration on the elliptic surface $E(1)= \mathbb{CP}^2 \# 9\barCP2$ with three sections which lie in the complement of $M(3,3,3)$. The complement of $M(3,3,3)$ in $E(1)$ is precisely the neighborhood of a regular fiber and these three sections. When we blow these sections down, we get the complement of $M(3,3,3)$ in $S_3$ which is a disk bundle over torus with Euler number +3. On the other hand the minimal resolution of $x^3+y^3+z^3 =0$ is a disk bundle over torus with Euler number -3. Hence $\xres$ is twisted $S^2$ bundle over $T^2$.

For general $p$ the argument goes as follows. We perturb $\Phi$ to a smooth curve $\widetilde \Phi$ of degree $p$ in $\mathbb{CP}^2$ and take the branched cover of $\mathbb{CP}^2$ branched along $\widetilde \Phi$. We get $\widetilde \pi: \xmf \to \mathbb{CP}^2$, so that $\widetilde \pi^{-1} (B)$ is diffeomorphic to $M(p,p,p)$. The complement $\mathbb{CP}^2 \setminus B$ is the neighborhood of a complex line $L$ which is a disk bundle over $S^2$ with Euler number +1. Now we look at the preimage of $\mathbb{CP}^2 \setminus B$ under $\widetilde \pi$. The branch curve $\widetilde \Phi$ intersects $L$ at $p$ points. Hence $\widetilde \pi^{-1}(L)$ is the branched cover of $S^2$ branched along $p$ points which implies $\widetilde \pi^{-1}(L)$ is a closed surface of genus $g= (p-1)(p-2)/2$. (We choose $B$ small enough so that it is away from the $p$ intersection points of $L$ and $\tilde\Phi$). The self intersection of $\widetilde \pi^{-1}(L)$ is $+p$, because the self intersection of $L$ is +1. Now $\widetilde \pi^{-1} (\mathbb{CP}^2 \setminus B)$ is a neighborhood of $\widetilde \pi^{-1}(L)$ which is a disk bundle over genus $g$ surface with Euler number $+p$. If we remove $M(p,p,p)$ from $S_p = \xmf$ and glue $\mmin$, this corresponds to taking the double of a disk bundle over a surface of genus $g$ with Euler number $p$. The result is an $S^2$ bundle over a surface of genus $g$. The bundle is twisted if $p$ is odd, untwisted if $p$ is even. 

\begin{corollary}
GC blow-up is the same as symplectic sum with $S_p$.
\end{corollary}
\begin{proof}
Recall that in order to do GC blow-up one needs to find an embedded symplectic surface $\Sigma$ of genus $g = (p-1)(p-2)/2$ with self intersection $-p$. Then a regular neighborhood $N(\Sigma)$ is symplectomorphic to $\mmin (p,p,p)$. On the other hand $S_p$ contains a symplectic surface  $\Sigma'$ with the same genus but opposite self intersection and the complement is the Milnor fiber $M(p,p,p)$. It is clear that removing $N(\Sigma)$ and gluing $M(p,p,p)$ is the same as identifying complements of $\Sigma$ and $\Sigma'$ as in symplectic sum.
\end{proof}

\begin{corollary}
For every positive integer $p$, the Lefschetz pencil over $S^2$ with fiber genus $(p-1)(p-2)/2$ described by the mapping class group relation given in Equation \ref{eq:rel}, has total space $S_p$. 
\end{corollary}

\section{Relationship with the chain surgery}
\label{Relationship with the chain surgery}
In this section we briefly discuss the relation between our surgeries and the chain surgery. Recall the chain relation in the mapping class group. Suppose we have $m$ curves $\alpha_1, \cdots, \alpha_m$ lying on a surface $F$ which satisfy the following intersection data
\begin{equation}
\# \alpha_i \cap \alpha_j = 1,  \mathrm{if } \; i-j=1, \quad
\# \alpha_i \cap \alpha_j = 0,  \mathrm{if }\; i-j \geq 2.
\end{equation}
We will call the system $\{\alpha_1, \cdots, \alpha_m\}$ an $m$-chain. Clearly an $m$-chain is $1\times m$ packing in our previous terminology. Suppose $m$ is even. Write $m=2g$. Then a neighborhood of $\bigcup_{i=1}^m \alpha_i$ in $F$ is homeomorphic to a surface $\Sigma_g^1$ with genus $g$ and with one boundary component. Let $t_i$ denote the right handed Dehn twist along $\alpha_i$ and let $t_{\delta}$ denote the right handed Dehn twist along the boundary parallel curve $\delta$. The chain relation is given by the following equation.

\begin{equation}
(t_1 \cdots t_{2g})^{4g+2} = t_{\delta}
\label{Chain1}
\end{equation}

In the left hand side of Equation \ref{Chain1}, the monodromy is along a $1\times 2g$ packing raised to the power $4g+2$, so this factorization describes the Lefschetz fibration on the Milnor fiber $M(2,2g+1, 4g+2)$.  

We first resolve the plane curve singularity $x^{2g+1}+y^{4g+2}$. The resolution graph with the multiplicities are shown in the first step of Figure \ref{ResAlgor} (see for example Ex.7.2.4b in \cite{GS}). Then we take double branched cover according to the algorithm in \cite{GS}, p.252 and we obtain a resolution graph of $x^{2g+1}+y^{4g+2}+z^2=0$ as in the second step of Figure \ref{ResAlgor} which is not minimal. Then we blow it down and get genus $g$ surface of self intersection $-1$.

\begin{figure}[h] 
		\includegraphics[width=.8\textwidth]{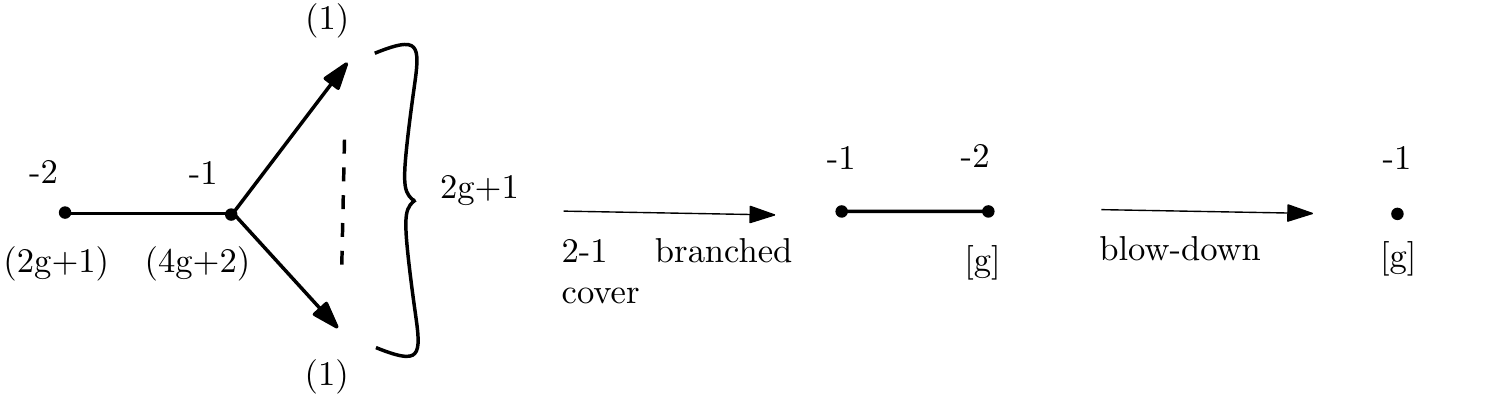}
		\caption{}
		\label{ResAlgor}
	\end{figure}


When $m$ is odd we write $m=2g+1$ and a neighborhood of $\bigcup_{i=1}^m \alpha_i$ is homeomorphic to $\Sigma_g^2$ with genus $g$ and with two boundary components. The chain relation in this case is 

\begin{equation}
(t_1 \cdots t_{2g+1})^{2g+2} = t_{\delta_1} t_{\delta_2}
\label{Chain2}
\end{equation}

Likewise, the left hand side of Equation \ref{Chain2} corresponds to the Milnor fiber $M(2,2g+2, 2g+2)$. The right hand side describes a Lefschetz fibration over the disk bundle over a surface of genus $g$ with self intersection $-2$. This is exactly the resolution of the singularity $x^2+y^{2g+2}+z^{2g+2}$ which we find similarly as above.

In any case, we see that chain substitution corresponds to exchanging the minimal resolution with the Milnor fiber of a Brieskorn singularity. One can also prove that this operation can be interpreted as a symplectic sum.

\bibliography{References}
\bibliographystyle{amsalpha}

\end{document}